\documentclass[11pt]{amsart}
\usepackage{latexsym}
\usepackage{amsfonts,amssymb,amsmath,amsthm,amscd}
\usepackage[mathscr]{eucal}

\newcommand{\I}{\mbox{${\mathbb I}$}}

\newcommand{\R}{\mbox{${\mathbb R}$}}

\newcommand{\Ha}{\mbox{${\mathcal H}$}}
\newcommand{\Ly}{\mbox{${\mathcal L}$}}
\newcommand{\Qa}{\mbox{${\mathcal Q}$}}
\newcommand{\tr}{{\rm tr}}
\newcommand{\ric}{{\rm Ric}}

\makeatletter
\def\numberwithin#1#2{\@ifundefined{c@#1}{\@nocnterrr}{%
  \@ifundefined{c@#2}{\@nocnterr}{%
  \@addtoreset{#1}{#2}%
  \toks@\expandafter\expandafter\expandafter{\csname the#1\endcsname}%
  \expandafter\xdef\csname the#1\endcsname
    {\expandafter\noexpand\csname the#2\endcsname
     .\the\toks@}}}}
\makeatother
\numberwithin{equation}{section}
\newtheorem{thm}[equation]{Theorem}
\newtheorem{lemma}[equation]{Lemma}
\newtheorem{prop}[equation]{Proposition}

\newtheorem{cor}[equation]{Corollary}
\newtheorem{ex}[equation]{Example}

\newtheorem{rem}[equation]{Remark}
\newenvironment{rmk}{\begin{rem} \em}{\end{rem}}
\begin{document}

\title{Non-k\"ahler Expanding Ricci Solitons}
\author{Andrew S. Dancer}
\address{Jesus College, Oxford University, OX1 3DW, United Kingdom}
\email{dancer@maths.ox.ac.uk}
\author{McKenzie Y. Wang}
\address{Department of Mathematics and Statistics, McMaster
University, Hamilton, Ontario, L8S 4K1, Canada}
\email{wang@mcmaster.ca}
\thanks{The second author is partially supported by NSERC Grant No. OPG0009421}
\date{revised \today}

\begin{abstract}
We produce new examples of non-K\"ahler complete expanding gradient Ricci solitons
 on trivial vector bundles over a product of Einstein manifolds with positive scalar
curvature.
\end{abstract}

\maketitle


\bigskip
\setcounter{section}{-1}

\section{\bf Introduction}

A Ricci soliton consists of a complete Riemannian metric $g$ and a
complete vector field $X$ on a manifold which give a solution to the equation:
\begin{equation} \label{LieRS}
{\ric}(g) + \frac{1}{2} \,{\sf L}_{X} g + \frac{\epsilon}{2} \, g = 0
\end{equation}
where $\epsilon$ is a real constant and $\sf L$ denotes the Lie derivative.
The soliton is called {\em steady} if $\epsilon =0$, {\em expanding}
if $\epsilon > 0,$ and {\em shrinking} if $\epsilon < 0$. If $X$ is a
Killing field (including the case $X=0$), $g$ is then Einstein. Such a soliton
is called {\em trivial}. Our main focus in this paper will be on nontrivial solitons.

Most of the known examples of nontrivial Ricci solitons are K\"ahler, e.g., \cite{Ko},
\cite{Ca}, \cite{ChV}, \cite{Gu}, \cite{PTV}, \cite{FIK}, \cite{WZ}, \cite{PS},
\cite{ACGT}, \cite{DW1}, \cite{Ya}. In \cite{DW2} we produced a family of
non-K\"ahler steady solitons generalising examples of Bryant and Ivey \cite{Bry}, \cite{Iv}.

In this article we produce expanding analogues of the examples in
\cite{DW2}; as in that paper our examples are warped products (over an
interval) with an arbitrary number $r$ of factors. In fact the
underlying manifolds are the same as in the steady case: trivial
vector bundles over a product of Einstein manifolds with positive
scalar curvature.  Since one can easily arrange for the dimensions of
the total spaces of the vector bundles to be odd, our solitons will in
general not be K\"ahler. Regarding earlier work, for $r=1$, a
discussion can be found in \S 5, Chapter 1 of \cite{Cetc}. The case
$r=2$ has been studied by Gastel and Kronz \cite{GK} using somewhat
different methods (related to those of C. B\"ohm \cite{B} in the
Einstein case (cf Remark \ref{Bohm})).

As is true for most known examples, the above solitons are of
{\em gradient type}, that is, we have $X = {\rm grad} \; u$ for a smooth
function $u$. Equation (\ref{LieRS}) then becomes
\begin{equation} \label{gradRS}
{\ric}(g) + {\rm Hess}(u) + \frac{\epsilon}{2} \, g = 0.
\end{equation}
As in the steady case, we produce our examples by analysing a reduction
of this equation to a system of ordinary differential equations.
The main difference between the steady and expanding cases is that the
Lyapunov function in the steady situation is no longer one in the
 expanding case.  Consequently, additional careful analysis of the behaviour
of the solution trajectories at the infinite end is necessary. In
particular, we need to use the theory of centre manifolds in studying their
$\omega$-limit sets.

Finally we note that there are numerous homogeneous expanding solitons
of non-gradient type on nilpotent Lie groups, as a result of the work of
Lauret \cite{La} and others. Our examples, however, will in general be inhomogeneous.

\section{\bf Multiple warped products}
We consider multiple warped products, that is, metrics of the form
\begin{equation} \label{metric}
dt^2 + \sum_{i=1}^{r} g_i^2(t)\,  h_i
\end{equation}
on $I \times M_1 \times ... \times M_r$\, where
$I$ is an interval in $\mathbb R$ and $(M_i, h_i)$ are Einstein manifolds
with positive Einstein constants $\lambda_i$. We let $d_i$ denote the (real)
dimension of $M_i$ and $n= d_1 + \dots + d_r$.

Moreover, the soliton potential $u$ is taken to be a function of $t$ alone.

The shape operator and Ricci endomorphism on the hypersurfaces of constant
$t$ are now given by
\begin{eqnarray*}
L_t &=& {\rm diag} \left( \frac{\dot{g_1}}{g_1} \,\I_{d_1}, \cdots,\frac{\dot{g_r}}{g_r} \,\I_{d_r}\right)  \\
{\rm Ric}^{\sharp}_t &=& {\rm diag} \left( \frac{\lambda_1}{g_1^2} \, \I_{d_1}, \cdots,\frac{\lambda_r}{g_r^2}\,
 \I_{d_r}\right)
\end{eqnarray*}
where $\I_m$ denotes the identity matrix of size $m$.

Motivated by \cite{Cetc}, we introduce new variables
\begin{eqnarray}
X_i &=& \frac{\sqrt{d_i}}{( -\dot{u} + {\tr} L)} \frac{\dot{g_i}}{g_i} \label{def-Xi},\\
Y_i &=& \frac{\sqrt{d_i \lambda_i}}{g_i} \frac{1}{(-\dot{u} + \tr{L})} \label{def-Yi},
\end{eqnarray}
for $i=1, \ldots, r$, and
\begin{equation}  \label{def-W}
W = \frac{1}{-\dot{u} + {\tr L}}.
\end{equation}

It is also convenient to introduce a new independent variable $s$ defined by
\begin{equation} \label{st}
\frac{d}{ds} := W \frac{d}{dt} =
\frac{1}{(-\dot{u} + {\rm tr} L)} \frac{d}{dt}.
\end{equation}
We use a prime ${ }^{\prime}$ to denote differentiation with respect to $s$.

In these variables the Ricci soliton system becomes the following equations
\begin{eqnarray}
X_{i}^{\prime} &=& X_i \left( \sum_{j=1}^{r} X_j^2 -1 \right) +
\frac{Y_i^2}{\sqrt{d_i}} + \frac{\epsilon}{2} \left( \sqrt{d_i}-X_i \right)
W^2 \, , \label{eqnX} \\
Y_{i}^{\prime} &=&  Y_i \left( \sum_{j=1}^{r} X_j^2 -\frac{X_i}{\sqrt{d_i}}
-\frac{\epsilon}{2}W^2 \right) \, , \label{eqnY} \\
W^{\prime} &=& W \left( \sum_{j=1}^{r} X_j^2 - \frac{\epsilon}{2}W^2  \right)
\label{eqnW}
\end{eqnarray}
for $i=1, \ldots, r$. The constant $\epsilon$ above is taken to be positive
as we are dealing with the case of expanding solitons. Note that the system
is invariant under the transformations $Y_i \mapsto -Y_i$ or $W \mapsto -W$.

\medskip
Conversely, if we have a solution of the above system, we may recover $t$ and the $g_i$
from
\begin{equation} \label{def-tgi}
dt = W \,\, ds, \ \ \ \ \
{g}_i = W \frac{\sqrt{d_i \lambda_i}}{Y_i}\,
\end{equation}
 The soliton potential $u$ is recovered from integrating
\begin{equation} \label{def-u}
 \dot{u} = \tr(L) - \frac{1}{W},
\end{equation}
where $\tr(L)$ is calculated using
\begin{equation} \label{logdiff-gi}
 \frac{\dot{g_i}}{g_i} = W^{-1} \frac{X_i}{\sqrt{d_i}}\, .
\end{equation}

As in the steady case the equations admit a conservation law
\begin{equation} \label{cons}
\Ly - \epsilon \left( u - \frac{(n-1)}{2} \right) W^2 = C W^2
\end{equation}
where $C$ is a constant and $\Ly = \sum_{j=1}^{r} (X_j^2 + Y_j^2) -1$.
In fact we showed in \cite{DW1} (cf Remark 2.8) and \S 1 of \cite{DW2}
that this is a general feature of gradient Ricci soliton equations of
cohomogeneity one type.

\section{\bf Trajectories of the equations}
A non-Einstein expanding soliton cannot be compact \cite{Per}, so, as
in the steady case \cite{DW2}, we shall construct complete non-compact
soliton metrics where there is a smooth collapse at one end onto a
lower-dimensional submanifold. This can be achieved if we take one
factor, say $M_1$, to be a sphere $S^{d_1}$, so the submanifold is
then  $M_2 \times \cdots \times M_r$. We can take the end to
correspond to $t=0$ without any loss of generality, With the
normalization $\lambda_1 = d_1 -1$, the boundary conditions for the
soliton solution to be $C^2$ are the existence of the following
limits:
\begin{equation} \label{bdy0}
g_1(0)=0   \;\; : \;\; g_i(0) = l_i \neq 0 \; (i > 1),
\end{equation}

\begin{equation} \label{bdy1}
\dot{g_1}(0)=1 \;\; : \;\; \dot{g_i}(0) = 0 \; (i > 1),
\end{equation}

\begin{equation} \label{bdy2}
\ddot{g}_{1}(0) =0 \;\; : \;\; \ddot{g_i}(0) \;\, {\rm  finite} \;
(i > 1),
\end{equation}

\begin{equation} \label{bdyu}
u(0), \, \, \dot{u}(0)=0, \, \, \ddot{u}(0) \, \, {\rm \, finite}.
\end{equation}

The regularity arguments of \cite{DW2} show that, if in addition the
third derivatives of $g_i$ tend to finite limits at $t=0$, the soliton
is in fact smooth.

With the above remarks in mind, we consider trajectories emanating from the
critical point of (\ref{eqnX})-(\ref{eqnW}) given by
\begin{equation} \label{critpoint}
W=0, \, X_1 = \beta, \;\; Y_1 = \hat{\beta}, \;\; X_i =Y_i=0  \,\, (i > 1)
\end{equation}
where $\beta = \frac{1}{\sqrt{d_1}} $ and $\hat{\beta} = +\sqrt{1- \beta^2}$.
This critical point lies on the unit sphere in $XY$-space.

\begin{rmk} \label{eqpt}
One can show that in the expanding case the only critical points of the
flow that do {\em not} lie on the unit sphere in $XY$-space with $W=0$
are

$($i$)$ the origin $O:=(W,X,Y)=(0,0,0),$

$($ii$)$ the points $E_{\pm}:=(W,X_i,Y_i) = (\pm\sqrt{\frac{2}{n\epsilon}},
   \frac{\sqrt{d_i}}{n}, 0)$ where $n = \sum_{j=1}^{r} d_j$.
\end{rmk}

We return now to the critical point (\ref{critpoint}).
Linearising about this point gives a system whose matrix
has a $2 \times 2$ block
\[
\left( \begin{array}{cc}
3 \beta^2 -1 & 2 \beta \hat{\beta} \\
\beta \hat{\beta} & 0
\end{array} \right)
\]
corresponding to $X_1, Y_1$ : the remaining entries are diagonal, with
$\beta^2$ and $\beta^2 -1$ occurring $r$ and $r-1$ times respectively.
(The extra $\beta^2$ term comes from the diagonal $W$ entry in the matrix.)

The eigenvalues are therefore
$\beta^2$ ($r$ times),  $\beta^2 -1$ ($r$ times), and $2 \beta^2$. When
$d_1 > 1$ the critical point is hyperbolic. We will assume this for the remainder
of the paper.

\medskip
We will parametrise trajectories emanating from this critical point
so that the critical point corresponds to $s=-\infty$.

We see
there is an $r$-parameter family of trajectories lying in the unstable
manifold of this critical point and not lying in the unit sphere in $XY$-space.

The equation (\ref{eqnY}) for $Y_i$ shows that $Y_i$ is never zero
unless it is identically zero. We note for future reference that the
sets $\{Y_i = 0 \}$ are closed invariant subsets of the flow.
Also the system (\ref{eqnX})-(\ref{eqnW}) and our initial conditions
are invariant under the symmetries $Y_i \mapsto -Y_i$ for $i > 1$.
Hence we may take $Y_i$ to be everywhere positive for $i>1$. Since
$\lim_{s \rightarrow -\infty} Y_1 = \hat{\beta} > 0$ we may
assume that $Y_1 > 0$  as well.

Similarly, the equation for $W$ shows that $W$ is either never zero or
identically zero. (Again, this means that $\{W =0\}$ is a flow invariant
subset.) Now the conservation law (\ref{cons}) shows the latter
cannot happen unless the trajectory is contained in the unit sphere
$\sum_{j=1}^{r} (X_j^2 + Y_j^2) =1$ and we are back in the Ricci flat case.
Having chosen $W$ to be nonzero, the symmetry of our equations under
$W \mapsto -W$ implies that we can take $W$ positive for all time.

\begin{lemma} \label{Xpos}
The variables $X_i$ are positive for all finite values of $s$.
\end{lemma}
\begin{proof}
  In Eqn (\ref{eqnX}) the coefficient of $X_i$ is negative
  for large negative $s$, say for $s \leq s^*$. (In fact, it tends to
  $\beta^2 -1 = \frac{1}{d_1} -1.$)
  If $X_i$ is non-positive at some $s_0 \leq s^*$, the equation
  (\ref{eqnX}) shows
  that $X_i^{\prime}(s_0)$ is positive, hence these inequalities hold
  on $(-\infty, s_0]$ and $X_i$ is negative and bounded away from zero
  as $s$ tends to $-\infty$, a contradiction to our choice of
  critical point.

Hence $X_i$ is positive on $(-\infty, s^*]$. But (\ref{eqnX})
shows that $X_i^{\prime}$ is positive if $X_i$ vanishes, so in fact
$X_i$ stays positive for all time.
\end{proof}

The following equations are easily deduced from (\ref{eqnX})-(\ref{eqnW}):
\begin{lemma}
We have the equations:
\begin{equation} \label{WY}
\left(\frac{W}{Y_i} \right)^{\prime} = \frac{X_i}{\sqrt{d_i}}
\left(\frac{W}{Y_i} \right),
\end{equation}
\begin{equation} \label{XY}
\left( \frac{X_i}{Y_i^2} \right)^{\prime} = \frac{X_i}{Y_i^{2}} \left(
-1 - \sum_{j=1}^{r} X_j^2 + \frac{2 X_i}{\sqrt{d_i}} + \frac{\epsilon}{2}W^2
 \right) + \frac{1}{\sqrt{d_i}} + \frac{\epsilon \sqrt{d_i}}{2}
\frac{W^2}{Y_i^2}.
\end{equation}
\end{lemma}

Eqn.(\ref{WY}) now gives the following result.
\begin{cor} \label{WYlim}
For each $i$, the function $\frac{W}{Y_i}$ is monotonic increasing: so
there exist $\mu_i \in [0, \infty)$ and $\sigma_i \in (0, \infty]$
such that $\lim_{s \rightarrow -\infty} \frac{W}{Y_i} = \mu_i$ and
$\lim_{s \rightarrow \infty} \frac{W}{Y_i} = \sigma_i$.    \qed
\end{cor}

\begin{rmk} Note that $\mu_i = \lim_{t \rightarrow 0} \frac{g_i}{\sqrt{d_i \lambda_i}} $
by (\ref{def-Yi}) and we shall show in Prop \ref{gifinite} that $\mu_i > 0$
if $i > 1$. Of course $\mu_1 = 0$.
\end{rmk}

\section{\bf Long-time behaviour of the flow}

We recall from \cite{DW2} the following quantities (the first of which
has been referred to already in the discussion of the conservation law at the end of \S 1):
\begin{equation}
{\mathcal L} = \sum_{i=1}^{r} (X_i^2 + Y_i^2 ) -1,
\end{equation}
\begin{equation}
{\mathcal H} = \sum_{i=1}^{r} \sqrt{d_i} X_i.
\end{equation}
It is also convenient to consider
\begin{equation}
\Qa = \Ly + \frac{\epsilon}{2}(n-1) W^2.
\end{equation}
Note that, in view of our initial conditions, we have $\Ly, \Qa \rightarrow 0$
and $\Ha \rightarrow 1$ as $s \rightarrow -\infty$.

Recall that in the steady ($\epsilon=0$) case, $\Ly$ was a Lyapunov function,
since it had the same sign as its derivative. In the current expanding case
this need no longer be true.

In fact we can calculate that our quantities satisfy the equations
\begin{equation} \label{eqnL}
{\mathcal L}^{\prime} = 2 {\mathcal L} \left( \sum_{i=1}^{r} X_i^2
- \frac{\epsilon}{2} W^2 \right) + \epsilon W^2 \left({\mathcal H} -1 \right),
\end{equation}
\begin{equation} \label{eqnH}
\left({\mathcal H} -1 \right)^{\prime} =
\left( \sum_{j=1}^{r} X_j^2 -1 -\frac{\epsilon}{2} W^2 \right)
({\mathcal H}-1 ) + \Qa,
\end{equation}
\begin{equation} \label{eqnQ}
\Qa^{\prime} =\epsilon W^2 \left( {\mathcal H} -1 \right) +
2 \left(\sum_{j=1}^{r} X_j^2 - \frac{\epsilon}{2} W^2 \right) \Qa.
\end{equation}

We have, therefore, equations of the form
\begin{equation}\label{eqnHQ1}
(\Ha-1)^{\prime} = f_1 (\Ha-1) + f_2 \Qa
\end{equation}
\begin{equation} \label{eqnHQ2}
\Qa^{\prime}     = f_3 (\Ha-1) + f_4 \Qa,
\end{equation}
where (for finite values of $s$) the $f_i$ are real analytic
and $f_2$ and $f_3$ are positive.

Repeated differentiation of this system shows that if $\Ha-1$ and
$\Qa$ both vanish at some $s_0,$ then they vanish there to all orders
and hence are identically zero.  These equations
further imply that if $\Ha = 1$ and $\Qa < 0$ at $s_0$ then
$\Ha^{\prime}(s_0) <0$, while if $\Qa=0$ and $\Ha < 1$ at $s_0$,
then $\Qa^{\prime} (s_0) < 0$.

\begin{prop} \label{prepbounds}
If $\Ha < 1$ and $\Qa < 0$ for large negative $s$, then these inequalities
hold for all finite $s$.
\end{prop}
\begin{proof}
If the inequalities do not hold for all time, consider the first point
$s_0$ at which one fails. This leads to one of the three possibilities
above, all of which give a contradiction.
\end{proof}

Next we note that $\Qa < 0$ for all $s \leq s^*$ implies that $\Ha < 1$
for all $s\leq s^*$. Indeed if we substitute our conservation law (\ref{cons})
into the above equation for $\Ha^{\prime}$ and simplify, we obtain
$$ \Ha^{\prime} =(1-\Ha)\left(\sum_{j=1}^r \,
Y_i^2  + \frac{\epsilon}{2} \, nW^2 \right) +
       {\Ha}W^2(C + \epsilon u). $$
But in terms of $\Qa$, the conservation law is
$$ \Qa = W^2 (C + \epsilon u). $$
So if at some $s_0 \leq s^*$ we have $\Ha \geq 1$, then $\Ha$ would be monotone
decreasing to the left of $s_0$, contradicting $\lim_{s \rightarrow -\infty} \Ha = 1$.

\begin{prop} \label{HQbounds}
There is an $r$-parameter family of our trajectories
with $\Ha < 1$ and $\Qa < 0$ for all finite $s$.
\end{prop}
\begin{proof} In view of Prop \ref{prepbounds} and the ensuing discussion,
it suffices to observe that there is an $r$-parameter family of initial
directions at our critical point (\ref{critpoint}) along which $\Qa$ is
strictly decreasing. These are given by tangent vectors to the unstable
manifold which have a negative inner product with the eigenvector
$(0, 2\beta, 0, \cdots, 0, \hat{\beta}, 0, \cdots, 0)$ of the linearization
(associated to the eigenvalue $2\beta^2$).
\end{proof}

\begin{cor} \label{bounds}    We have
\begin{enumerate}
\item $\Ly < 0$ for all finite $s$, so the trajectory
stays in the  region $\sum_{j=1}^{r} (X_j^2 + Y_j^2) <1$.

\item $W < \sqrt{\frac{2}{\epsilon}}$ for all finite $s$.

\item All variables $W, X_i, Y_i$ are bounded so the flow exists
for $s \in (-\infty, \infty)$.
\end{enumerate}
\end{cor}

\begin{proof}
Now (1) follows from Prop \ref{HQbounds} since $\Ly < \Qa.$ If
$W \geq \sqrt{\frac{2}{\epsilon}}$ at some $s_0$,
then Eqn.(\ref{eqnW}) implies that $W^{\prime}(s_0) < 0$. Hence $W$ is
monotone decreasing to the left of $s_0$, contradicting
$\lim_{s \rightarrow -\infty} W = 0.$ This proves (2), and (3) follows
immediately. We note that a better upper bound for $X_i$ follows from
using $\Ha < 1$.
\end{proof}

\begin{rmk} \label{einstein}
It follows from equations (\ref{eqnHQ1})-(\ref{eqnHQ2})
 that the $2r-1$-dimensional submanifold
${\mathcal S}$ of phase space defined by
$\Ha =1, \Qa =0$ is invariant under the flow. As
\[
u^{\prime} = W \dot{u} = W \,{\rm tr}( L)  -1 = \sum_{i=1}^{r} \sqrt{d_i} X_i -1 =
\Ha -1
\]
we see that trajectories in $\mathcal S$ give Ricci solitons with constant potential $u$,
that is, Einstein metrics.
\end{rmk}

\medskip
We now return to the trajectories of Proposition \ref{HQbounds}, and in
particular to the analysis of their long-time behaviour. Let $\gamma$ denote
such a trajectory (with $W,  Y_i$ positive).

Recall that the $\omega$-limit set of the trajectory $\gamma$ is the set
\[
\Omega = \{ (W^*,X^*,Y^*) : \exists \, s_n
\rightarrow +\infty \; {\rm  with} \;
\gamma(s_n) = (W(s_n),X(s_n),Y(s_n)) \rightarrow (W^*,X^*,Y^*) \}.
\]
As our trajectory actually lies in a compact set, we know
(\cite{Pk} \S 3.2) that $\Omega$ is a compact, connected, non-empty set
that is invariant under the flow of our equations (both forwards
and backwards). Moreover,
by Cor. \ref{WYlim}, $\Omega$ is contained in the set
\[
\left\{ \frac{W}{Y_i}  = \sigma_i, \, \,  \; i=1,\ldots, r \right\}.
\]
If $\sigma_i=+\infty$, this condition is equivalent (since $W$ is
bounded) to $Y_i=0$. If $\sigma_i$ is finite, then we claim that
$\Omega$ still lies in $Y_i = 0$. To see this, note first that  flow-invariance
and (\ref{WY}) show that $\Omega$ is contained in $X_i=0$. Now
flow-invariance and Eqn.(\ref{eqnX}) show we must have
$Y_i=W=0$. We conclude that
\[
\Omega \subset \bigcap_{i=1}^{r} \, \{ (W,X,Y) :  Y_i=0 \},
\]
which is also a flow invariant set. As well, by Proposition \ref{HQbounds},
Lemma \ref{Xpos}, and Corollary \ref{bounds}, we also have
\begin{equation} \label{flowbox}
 \Omega \subset \left\{ 0 \leq X_i \ (1 \leq i \leq r), \,
 0 \leq W \leq \sqrt{\frac{2}{\epsilon}}, \ \
 \Ha \leq 1,  \ \ \Qa \leq 0,  \right\}.
\end{equation}

\begin{rmk} We shall see in Lemma \ref{limsigma} that $\sigma_i$
is always infinite.
\end{rmk}

We will now show that $\Omega$ contains the origin, and from this, with some more work,
deduce the convergence of the trajectory $\gamma$ to the origin.
To this end we will consider the flow restricted to $\Omega$.

\medskip
\begin{prop} \label{Omegazero}
If $\Omega$ contains a point with zero $W$-coordinate, then $\Omega$
 contains the origin. 
\end{prop}

\begin{proof}
If $\Omega$ contains a point $p$ whose $W$-coordinate is zero,
let $\phi$ denote its trajectory. Then along it, we have $W=0$
and (from above) $Y_i =0$. Now
the $X_i$ evolve by
$$X_i^{\prime} = X_i \left(\sum_{k=1}^{r} X_k^2 -1 \right). $$
This implies that $G:=\sum_{k=1}^{r} X_k^2$ satisfies the equation
$$G^{\prime} =2G(G-1).$$
Moreover $G \leq 1$ since $\Qa \leq 0$. If $G \equiv 0$, then $p$ {\em is} the origin.
If $G$ is not identically $1$, then upon integrating the
equation we see that $\phi$ flows into the origin. Hence $\Omega$,
being closed and flow invariant, must contain the origin. If $G \equiv 1$,
then $p$ is a critical point of our vector field lying on the unit sphere
in $XY$ space (cf Remark \ref{eqpt}). We will show by examining the
linearisation of (\ref{eqnX})-(\ref{eqnW}) at $p$ that any trajectory
(which is not a critical point) must leave any small neighbourhood of $p$
in finite time and cannot return. This contradicts $p \in \Omega$.

Let $p=(0, p_1, \cdots, p_r, 0, \cdots, 0)$ with $\sum_{i=1}^r \, p_i^2 =1$.
Setting $X_i = x_i + p_i, \, Y_i = y_i$ and $W=w$, we have
\begin{eqnarray*}
x_i^{\prime} &=& 2p_i \sum_{j=1}^r \, p_j x_j + \cdots \\
y_i^{\prime} &=& (1-\frac{p_i}{\sqrt{d_i}})\, y_i + \cdots \\
w^{\prime} &=& w + \cdots
\end{eqnarray*}
where $\cdots$ denote terms of higher order. Hence the zero eigenspace
of the linearisation coincides with the tangent space of the critical
submanifold $\{W=0, Y_i = 0 \, (1 \leq i \leq r), \sum_i \, X_i^2 = 1 \}$
at $p$. The remaining eigenvalues of the linearisation are $2$ (with
eigenvector $p$), $1$ , and $1-\frac{p_i}{\sqrt{d_i}} > 0, (1 \leq i \leq r)$.
The stable manifold at $p$ therefore reduces to a point and the zero
eigenvalues correspond to directions tangent to the
submanifold of fixed points of the flow. This proves our claim.
\end{proof}

\medskip
In view of the preceding result, we now consider the
case when $\Omega$ does not contain any points
with $W$-coordinate equal to $0$. Since $\Omega$ is compact, the
function $W$ is bounded below by a positive constant $w_0$ on
$\Omega$.  We again consider the trajectory $\phi$ of a point $p \in
\Omega$ and denote by $\Omega_{\phi}$ its $\omega$-limit set.  Note
that $\Omega_{\phi} \subset \Omega$ (since $\Omega$ is closed and
flow-invariant) and $\Omega_{\phi}$ is also non-empty, compact,
connected and flow-invariant.  In analysing this set we may assume, by
(\ref{eqnX}) and the fact that $W$ is nonzero, that $X_i > 0, \, 1
\leq i \leq r,$ along $\phi$.

Letting $\tilde{X}_{i} = X_i/\sqrt{d_i},$  we have along $\phi$
\[
\left(\frac{\tilde{X_i}}{\tilde{X_j}} \right)^{\prime} =
\frac{\epsilon W^2}{2 \tilde{X}_{j}} \left( 1- \frac{\tilde{X}_i}{\tilde{X}_j} \right),
\]
where we have used the fact that $Y_i = 0, 1 \leq i \leq r$.
Hence $\frac{\tilde{X_i}}{\tilde{X_j}}$
 is either monotonic decreasing as $s \rightarrow +\infty$ to a limit
$\lambda_{ij} \geq 1$, or monotonic increasing to a limit $\lambda_{ij} \leq 1$.
(This includes the possibility that $\frac{\tilde{X_i}}{\tilde{X_j}}$ is
identically 1.) Letting $\tau_{ij} =
\frac{\sqrt{d_i}}{\sqrt{d_j}}\lambda_{ij}$, we see that (since $X_j$ are bounded)
\[
X_i - \tau_{ij} X_j = X_j \, \left( \frac{X_i}{X_j} - \tau_{ij} \right) \rightarrow 0
\;\;{\rm as} \;\; s \rightarrow +\infty.
\]
Hence $\Omega_{\phi}$ must lie in
the subspace
\[
\{ X_i - \tau_{ij} X_j =0, \;\; \ 1 \leq i,j \leq r \}
\]
for some constants $\tau_{ij}$ (depending on $\phi$).
As
\[
(X_i -\tau_{ij}X_j)^{\prime} = \left(X_i -\tau_{ij} X_j \right) \left(\sum_{k=1}^{r} X_k^2 -1
 -\frac{\epsilon}{2} W^2 \right) + \frac{\epsilon}{2} W^2 (\sqrt{d_i} -
\tau_{ij} \sqrt{d_j})
\]
and $ W>0$, we must have $\tau_{ij} = \sqrt{d_i}/\sqrt{d_j}$, i.e.,
 $\lambda_{ij}=1$ for all $i,j$.

Therefore, $\Omega_{\phi}$ is contained in a plane given by
$X_i = \frac{\sqrt{d_i} X_1}{\sqrt{d_1}},  \; i=2, \ldots, r$.
Parametrising this plane by $X_1, W$, we can write the flow as:
\begin{eqnarray}
X_1^{\prime} & =& X_1 \left( \frac{n}{d_1} X_1^2 -1 \right) +
\frac{\epsilon}{2} \left(\sqrt{d_1} - X_1 \right) W^2  \label{resX1} \\
W^{\prime} &=& W \left( \frac{n}{d_1} X_{1}^2  - \frac{\epsilon}{2} W^2 \right) \label{resW}.
\end{eqnarray}
Note that in this $2$-plane, we have $\Ha = \frac{n}{\sqrt{d_1}} X_1$, so
the condition $\Ha \leq 1$ becomes $X_1 \leq  \sqrt{d_1}/n$.

In view of (\ref{flowbox}) and the fact that on $\Omega$ we also
have $W \geq w_0$  for some positive constant $w_0$, we only need to focus
on the flow in the rectangular box in the $X_1W$-plane given by
$${\mathcal B} := \left\{0 \leq X_1 \leq \frac{\sqrt{d_1}}{n},  \,\,
w_0 \leq W \leq \sqrt{\frac{2}{\epsilon}}  \right\},$$
which contains $\Omega_{\phi}$.
The only critical point of (\ref{resX1})-(\ref{resW}) lying in $\mathcal B$ is
$(X_1, W) = (\frac{\sqrt{d_1}}{n}, \sqrt{\frac{2}{n\epsilon}})$. which
lies on the edge of $\mathcal B$. (This corresponds, of course, to the critical
point $E_{+}$ of the full system discussed in Remark \ref{eqpt}, which
lies in $\{ \Ha=1, \Qa=0 \}$.)

Observe that above the line $X_1 = \sqrt{\frac{\epsilon  d_1}{2n}}W$ in
this region we have $W$ strictly decreasing while below this line $W$
strictly increases. On the line, $W^{\prime}$ is zero and $X_1^{\prime}=
X_1 (-1 + \frac{nX}{\sqrt{d_1}})$ is negative. On $X_1=0$ we have $X_1^{\prime}$
positive, while on $X_1 = \sqrt{d_1}/n$ we have $X_1^{\prime}$ negative
below the critical point $E_+$ and positive above it. One can also check that
on the line $W=\sqrt{\frac{2}{\epsilon}}$ we have $X_1^{\prime}> 0$
and $W^{\prime} < 0$. On the line $W=w_0$,(for $w_0$ small),  $X_1^{\prime}$
 goes from positive
to negative. As $W$ increases from $w_0$ to $\sqrt{\frac{2}{n\epsilon}}$
the zero of $X_1^{\prime}$ moves to the right and tends to $\sqrt{d_1}/n$.

We deduce from the above picture that the only trajectory that starts
in $\Omega_{\phi} \subset {\mathcal B}$ and remains in $\mathcal B$
must lie above the line $X_1 = \sqrt{\frac{\epsilon  d_1}{2n}}W$ and must
flow into $E_+$. Indeed, if we linearise the flow in the plane about $E_+$, we find that
the eigenvalues of the linearised operator are
$\frac{1}{2} ( -1 \pm \sqrt{1 + (8/n)})$, so that $E_+$ is a hyperbolic critical point.
Since $\Omega_{\phi}$ is connected and flow-invariant,
 we conclude that $\Omega_{\phi}$ consists
 of a trajectory which ultimately coincides with the part of the stable
manifold (curve) at $E_+$ lying in $\mathcal B$. In particular, there must be
a sequence of points on our first trajectory $\gamma$ which lie arbitrarily
close to this portion of the stable curve.

Next recall that $\Omega$, and hence $\Omega_{\phi}$, are contained in $\{\Qa \leq 0\}$.
The intersection of the quadric $\Qa = 0$ and our $2$-plane is the ellipse
$\frac{n X_1^2}{d_1} + \frac{\epsilon}{2}(n-1) W^2 =1$.
$E_+$ lies on this ellipse and we may take $\nabla Q |_{E_+}= (\frac{2}{\sqrt{d_1}},
(n-1)\sqrt{\frac{2 \epsilon}{ n}})$ as a normal vector to the quadric at $E_+$.
 A tangent vector to the stable curve which points out of the region
 $\mathcal B$ is $\xi:=(\sqrt{\frac{2d_1 \epsilon}{n}}, \frac{n}{n-1}(\frac{1}{2}-
\frac{1}{2}\sqrt{1 + \frac{8}{n}} - \frac{2}{n}))$. One can now check that
$\nabla Q|_{E_+} \cdot \xi <0$. Hence the stable curve and the quadric $\Qa =0$
lie on opposite sides of the tangent space at $E_{+}$ to the quadric. This contradicts
the last statement of the previous paragraph unless $\Omega_{\phi} =\{ E_+ \}$ for
every trajectory $\phi$ in $\Omega$.

\medskip
We have shown:

\begin{prop} \label{Omeganonzero1}
Suppose that all points of $\Omega$ have nonzero $W$-coordinate. Then for each
point $p \in \Omega$, the $\omega$-limit set $\Omega_{\phi}$ of the trajectory $\phi$
through $p$ is the singleton set $\{ E_+ \}$.  \qed
\end{prop}

\begin{lemma} \label{Etraj}
A trajectory of the full system $($\ref{eqnX}$)$-$($\ref{eqnW}$)$
lying in a compact set and having $\omega$-limit
set $\{E_+ \}$  must converge to $E_+$. Moreover, either it lies entirely in the
invariant submanifold ${\mathcal S}:=\{ \Qa=0, \Ha =1 \}$ or it eventually
lies $($for finite values of $s$$)$ in $\{\Qa < 0, \Ha  > 1 \}$ or $\{\Qa > 0, \Ha < 1\}$.
\end{lemma}
\begin{proof} Let $\Gamma$ denote such a trajectory. Convergence follows because
otherwise one can find $s_n \rightarrow +\infty$ with $\Gamma(s_n)$ lying outside
some open ball around $E_+$. But a subsequence of $\Gamma(s_n)$ would then converge
to a point in the $\omega$-limit set distinct from $E_+$, a contradiction.

Consider the linearisation at $E_+$ of the full system
(\ref{eqnX})-(\ref{eqnW}). One finds that the linearised operator
has only one positive eigenvalue $\frac{1}{2} ( -1 + \sqrt{1 + \frac{8}{n}})$;
the remaining eigenvalues are $-1/n$ ($r$ times), $-1$ ($r-1$ times) and
$\frac{1}{2} ( -1 - \sqrt{1 + \frac{8}{n}})$. Furthermore, the $-1/n$ eigenspace
consists of vectors of the form $(W, X, Y) = (0, 0, Y)$ while the $-1$ eigenspace
consists of vectors of the form $(W, X, Y)=(0, \xi, 0)$ where $\xi$ is orthogonal
to $\eta:=(\sqrt{d_1}, \cdots, \sqrt{d_r})$. An eigenvector for the remaining
negative eigenvalue is
$v:=(\frac{n}{n-1}(\frac{1}{2}(-1+\sqrt{1+ \frac{8}{n}})+ \frac{2}{n}),
-\sqrt{\frac{2 \epsilon}{n}} \eta, 0)$.

Notice that except for $v$, the rest of the eigenvectors with
negative eigenvalues are orthogonal to $\nabla \Qa$ and $\nabla \Ha$ at $E_+$.
So they actually span the tangent space at $E_+$ to $\mathcal S$. (In particular,
for the restricted flow on $\mathcal S$ the point $E_+$ is actually a sink.)
One checks that $v \cdot \nabla \Qa > 0$ and $v \cdot \nabla \Ha < 0$ at $E_+$.
Now our trajectory $\Gamma$ must ultimately lie in the stable manifold of $E_+$,
so any trajectory in the stable manifold at $E_+$ which does not lie
in $\{\Ha =1, \Qa =0 \}$ must lie in $\{\Qa < 0, \Ha  > 1 \}$
or $\{\Qa > 0, \Ha < 1\}$.
\end{proof}

\begin{cor} \label{Omeganonzero2}
Suppose that all points of $\Omega$ have nonzero $W$-coordinate. Then $\Omega$ is
a union of trajectories lying in the intersection of the invariant submanifold
${\mathcal S} :=\{\Qa=0, \Ha=1 \}$ with the invariant set $\{ Y=0 \}$. \qed
\end{cor}

We shall denote the invariant set $\{\Qa=0, \Ha =1, Y=0 \}$ by $\hat{\mathcal S}$.
Note that $\hat{\mathcal S}$, like $\mathcal S$, is compact because of the
condition $\Qa=0$.

We may now observe that on the set $\{\Ha =1, Y=0\}$  the quantities
$G:= \sum_{j=1}^{r} X_j^2$ and $J = G -\frac{\epsilon}{2} W^2$ satisfy

\[
G^{\prime} = 2(G-1)J
\]
and
\begin{equation} \label{eqnJ}
J^{\prime} = 2J(J-1).
\end{equation}
Moreover, if we in addition impose the condition $\Qa=0$, then we may rewrite $J$
as $J = \frac{nG -1}{n-1}$. The Cauchy-Schwartz inequality applied to the
condition $\Ha := \sum_{i=1}^{r} \sqrt{d_i} X_i =1$ shows that $G \geq \frac{1}{n}$
on $\Ha =1$, with equality if and only if $X_i =\frac{\sqrt{d_i}}{n}$ for all $i$.
Furthermore, the condition $\Qa =0$ implies that $\Ly \leq 0$ and hence $G \leq 1$.

Hence we find that on the set $\hat{\mathcal S}$ the function $J$ satisfies
$0 \leq J \leq 1$, with $J=0$ only at the points $E_{\pm}$ given by
$(W,X,Y)=(\pm \sqrt{\frac{2}{\epsilon n}}, \frac{\sqrt{d_1}}{n}, \cdots,\frac{\sqrt{d_r}}{n} ,0, \cdots)$,
and $J=1$ only at  the points with $Y=W=0$ and $\sum_{j=1}^{r} X_j^2 =1$.
Note that these latter points are the critical points of the flow discussed
in the proof of Proposition \ref{Omegazero}.

It now follows from equation (\ref{eqnJ}) that $J$ is a Lyapunov function on
$\hat{\mathcal S}$. Points on the plane $W=0$ in $\hat{\mathcal S}$ have $J=1$
and are fixed under the flow. Points in the regions $W >0$ and $W < 0$ flow
to the critical points $E_{+}$ and $E_{-}$ respectively.

We can now deduce:

\begin{prop} \label{closure}
 Let $\gamma$ be one of the $r$-parameter
  family of trajectories with $Y_i,W$ positive emanating from the
  critical point $($\ref{critpoint}$)$ and entering the region $\{ \Qa
  < 0 \}$. Then its $\omega$-limit set $\Omega$ contains the origin.
\end{prop}

\begin{proof}
  By Proposition \ref{Omegazero} we just have to consider the case
  when $\Omega$ contains no point with zero $W$-coordinate.  By
  Corollaries \ref{Omeganonzero2} and \ref{flowbox}  we know that
  $\Omega$ is contained in the region of $\hat{\mathcal S}$ with $W
  > 0$. But it is also closed and invariant under forwards and backwards flows. Our discussion above
shows that $\Omega = \{ E_+ \}$, as if it contained another point then
the backwards trajectory would converge to a critical point with $W=0$, which would thus have to
lie in $\Omega$.

So $\Omega$ consists of the critical point $\{ E_+ \}$. Now Lemma \ref{Etraj} shows that
the bounds established in Proposition \ref{HQbounds} are violated, so this case cannot occur.
\end{proof}

\medskip
Our next aim is to show that the trajectory $\gamma$ actually converges
to the origin.

The linearisation of our flow (\ref{eqnX})-(\ref{eqnW}) about the origin is
\begin{eqnarray*}
X_i^{\prime} &=& - X_i, \\
Y_i^{\prime}   &=&  0, \\
W^{\prime}   &=&  0,
\end{eqnarray*}
so the $0$-eigenspace is given by $X=0$. Following Glendinning \cite{Gl},
for example, we seek a centre manifold with local expression in the form
$X_i = h_i(Y,W)$, where we need $h_i$ and $D h_i$ to vanish at the origin.

We write
\[
X_i = \sum_{j,k} \, a_{jk}^{(i)} Y_j Y_k + \sum_{j}\, b_j^{(i)} Y_j W + c^{(i)} W^2 + \cdots
\]
where $\ldots$ refers to terms of higher order than quadratic order.  Now
$X_i^{\prime} = \sum_{k} \frac{\partial h_i}{\partial Y_k}
Y_{k}^{\prime} + \frac{\partial h_i}{\partial W} W^{\prime}$, while
$X_i^{\prime}$ can also be expressed using (\ref{eqnX}). Equating terms we
find the quadratic terms in the expression for $X_i$ are as follows:
\[
X_i = \frac{1}{\sqrt{d_i}} Y_i^2 + \frac{\epsilon \sqrt{d_i}}{2} W^2 + \cdots
\]
So $Y_i^{\prime} = Y_i ( -\frac{1}{d_i} Y_i^2 - \epsilon W^2 + \cdots)$
and $W^{\prime} = -\frac{\epsilon}{2} W^3 + \cdots,$
where $\cdots$ in the above refers to terms with total degree greater
than those displayed.

It follows from this that
$$ \left(\sum_{i=1}^r \, Y_i^2 + W^2 \right)^{\prime} = -2 \left( \sum_{i=1}^r \frac{1}{d_i} Y_i^4
     + \epsilon W^2 \sum_{i=1}^r \, Y_i^2 + \frac{\epsilon}{2} W^4 \right) + \mbox{\rm higher order terms}.
$$
In other words, $(\sum_{i=1}^r \, Y_i^2) + W^2$ is a Lyapunov function for
the local dynamics of our centre manifold near the origin.  Hence the origin
is asymptotically stable, i.e., on the centre manifold the flow near the origin
converges to the origin. Since the nonzero eigenvalues
of the linearisation of the full system are negative, the centre manifold
theorem yields

\begin{prop} \label{sink}
The origin is a sink for our flow. \qed
\end{prop}

Combining Prop \ref{sink} with Prop \ref{closure} (which showed that
the trajectory $\gamma$ approaches the origin arbitrarily closely), we have

\begin{thm} \label{omegalimit}
The trajectory $\gamma$ converges to the origin as $s$ tends to $+\infty$. \qed
\end{thm}

The above theorem allows us to deduce
information about the asymptotics of our metric
once we derive the following limit; the proof is similar to that for
Lemma 3.7 in \cite{DW2}.

\begin{lemma} \label{limXW}
We have $\lim_{s \rightarrow +\infty} \frac{X_i}{W^2} :=
\Lambda_i = \frac{1}{\sigma_i^2 \sqrt{d_i}} + \frac{\epsilon}{2} \sqrt{d_i}$.
\end{lemma}

\begin{proof}
Observe that $X_i/W^2$ satisfies the differential equation
\begin{equation} \label{eqnXW}
\left( \frac{X_i}{W^2} \right)^{\prime} =  \left(-1 -\sum_{j=1}^{r} X_j^2
 \right) \frac{X_i}{W^2} + \frac{Y_i^2}{W^2 \sqrt{d_i}} +
\frac{\epsilon}{2} (X_i + \sqrt{d_i})
\end{equation}
By Theorem \ref{omegalimit}, the coefficient $F_i$ of $\frac{X_i}{W^2}$
tends to $-1$ as $s \rightarrow +\infty$. The sum $K_i$ of the remaining
terms on the right hand side converges to $\Lambda_i$. In particular,
if $\frac{X_i}{W^2}$ tends to a finite limit $a,$ then its derivative
tends to $-a + \Lambda_i$, so $a$ must equal $\Lambda_i$.

Let $0 < \delta < \min(1, \Lambda_i)$, and pick $s^*(\delta)$ so that
$\sum_{j=1}^{r} X_j^2$ is less than $\delta$ and
$\Lambda_i -\delta < K_i < \Lambda_i + \delta$  for $s \geq s^*(\delta)$.
It follows that
if $\frac{X_i}{W^2} (s_0) \geq \frac{\Lambda_i + \delta}
{1 - \delta}$ for some $s_0 > s^*(\delta),$
then $\left(\frac{X_i}{W^2} \right)^{\prime} <0$ at $s_0$.
Similarly, if $\frac{X_i}{W^2}(s_0) \leq \frac{\Lambda_i - \delta}
{1+\delta}$, then $\left(\frac{X_i}{W^2} \right)^{\prime} >0$ at $s_0$.
So if $\frac{X_i}{W^2}$ enters the horizontal strip
$\frac{\Lambda_i - \delta}{1+\delta} < y < \frac{\Lambda_i + \delta}
{1-\delta}$ at some $s > s^*(\delta),$ it is trapped there.

On the other hand, if $\frac{X_i}{W^2}$ never enters the strip for
$s > s^*(\delta)$ then
it is either monotonic increasing to some number less than or equal to
$\frac{\Lambda_i - \delta}{1 + \delta}< \Lambda_i$,
or monotonic decreasing to some number greater than or equal to
$\frac{\Lambda_i + \delta}{1 - \delta} > \Lambda_i$. Both of these
outcomes contradict the discussion in the first paragraph.

So $\frac{X_i}{W^2}$ is indeed trapped in the strip for large enough $s$.
Since this conclusion holds for all $\delta$, the result follows.
\end{proof}

We therefore have an estimate
\begin{equation} \label{XWest}
\sum_{k=1}^{r}  X_k^2 \leq c W^4 \;\;\;{\rm for} \; s \; {\rm large}
\end{equation}
where $c$ is a positive constant. Hence, by (\ref{eqnW}), $W$ is
decreasing for sufficiently large $s$.

\begin{lemma} \label{complete}
The metric corresponding to our trajectory is complete at $s = +\infty$.
\end{lemma}
\begin{proof}
The geodesic distance to infinity is $\int^{\infty} W \; ds =
\int_{0} \frac{dW}{\frac{\epsilon}{2} W^2 -\sum_{k=1}^{r} X_k^2}$, which is infinite
\end{proof}

\begin{lemma} \label{limsigma}
The quantities $\Lambda_i$ equal $\frac{\epsilon \sqrt{d_i}}{2}$. Equivalently,
the quantities $\sigma_i = \lim_{s \rightarrow +\infty} \frac{W}{Y_i}$
equal  $\infty.$ Hence $\lim_{t \rightarrow +\infty} \frac{\dot{g}_i}{g_i W}
= \frac{\epsilon}{2}.$
\end{lemma}
\begin{proof}
By (\ref{eqnY}) and (\ref{eqnW}) we have
\[
\frac{W^{\prime}}{Y_i^{\prime}} = \frac{W}{Y_i}
\frac{(\sum_k \, X_k^2 -\frac{\epsilon}{2}W^2)}
{(\sum_k  X_k^2 - \frac{X_i}{\sqrt{d_i}} -\frac{\epsilon}{2} W^2)}.
\]
The right-hand side tends to
$ \frac{\sigma_i\epsilon}{\epsilon + (2 \Lambda_i/\sqrt{d_i})}$,
and $\Lambda_i$ is positive. But by L'H{\^o}pital's rule, when the
left-hand side tends to a limit this must be $\sigma_i$.
So $\sigma_i \in (0, \infty]$ cannot be finite.
The last assertion now follows immediately from (\ref{logdiff-gi}).
\end{proof}

In fact, (\ref{eqnW}) and the estimate (\ref{XWest}) show that for any
sufficiently small $\delta > 0$, there is some $\hat{s}$ and constants
$A_1$ and $A_2$ such that on $[\hat{s}, +\infty)$ we have
$$ \frac{1}{\epsilon s + A_1} \leq W^2 \leq \frac{1}{(\epsilon-2\delta)s + A_2}.$$
As $dt = W \; ds$, we deduce $s \sim \frac{\epsilon t^2}{4}$ so
$W \sim \frac{2}{\epsilon t}$. Integrating the relation
$\frac{\dot{g}_{i}}{g_i} = \frac{1}{ \sqrt{d_i}} \frac{X_i}{W} $ then shows
that for all sufficiently small $\delta> 0$ there is some $\hat{t}$ and
positive constants $a_1, a_2$ (depending on $\delta$) satisfying
$1-  \delta < \frac{a_2}{a_1} < 1$ such that on
$[\hat{t}, +\infty)$ we have
$$(a_1 (t-\hat{t}) + 1)^{1-\delta} < \frac{g_i(t)}{g_i(\hat{t})} < (a_2 (t-\hat{t}) + 1)^{1+ \delta}.$$

\begin{rmk} \label{asymp}
These estimates are consistent with our expectation
that the  metric has an asymptotically conical geometry.
Asymptotically conical behaviour has also been observed for the
known examples of K\"ahlerian expanding Ricci solitons.  By contrast,
the known examples of steady solitons are usually asymptotic
to a paraboloid or a circle bundle over a paraboloid.
\end{rmk}

\section{\bf The flow near the critical point}
To check smoothness at the collapsing submanifold, we must now analyse
the behaviour of our trajectory as $s$ tends to $-\infty$. Recall that
we have arranged for $X_1 \rightarrow \beta = \frac{1}{\sqrt{d_1}},
\, Y_1 \rightarrow +\sqrt{1- \beta^2},$ and the remaining
variables to tend to $0$. Most of the analysis below is analogous to
that for the steady case with the role of $\sqrt{{\mathcal L}/C}$
now played by $W$.

We first recall the following useful lemma from \cite{DW2}.

\begin{lemma} \label{FHKlemma}
Suppose a function $F$ satisfies a differential equation
\begin{equation} \label{eqnFHK}
F^{\prime} = H F + K
\end{equation}
where $H, K$ are functions tending respectively to finite limits
$h,  k $ as $s$ tends to $-\infty$, where $h<0$ and $k \neq 0$.

Then either $\lim_{s \rightarrow -\infty} F(s) = -\frac{k}{h}$
or $F$ tends to $\infty$ or $-\infty$ as $s$ tends to $-\infty$. Moreover
in the case of infinite limit $F$ is monotonic for sufficiently large
negative $s$.
\end{lemma}

\begin{lemma} \label{XYbounded} As $s \rightarrow -\infty$, $X_i/Y_i^2$ remains
bounded.
\end{lemma}

\begin{proof}
If $i=1$ the claim follows immediately from our initial conditions.

Now let $i>1$. Since $\lim_{s \rightarrow -\infty} (\sum_{j=1}^{r} X_j^2)  = \beta^2,$
$\lim_{s \rightarrow -\infty} X_i = 0 = \lim_{s \rightarrow -\infty} W,$
and $\lim_{s \rightarrow -\infty} \frac{W}{Y_i} = \mu_i
< \infty$, for any $0 < \delta < 1 -\beta^2$ we can find $s^*$ so that
for $s \leq s^*$ we have
$$ \frac{Y_i^2}{\sqrt{d_i}} + \frac{\epsilon}{2}(\sqrt{d_i}-X_i) W^2 <
       \frac{Y_i^2}{\sqrt{d_i}} + \frac{\epsilon}{2}(\sqrt{d_i} + \delta)(\mu_i + \delta)^2
 Y_i^2$$
and $Y_i^{\prime} > 0$, and also $\sum_{j=1}^{r} X_j^2 -1 < \beta^2 -1 + \delta < 0$.

We claim that on $(-\infty, s^*]$ we must have
$$\frac{X_i}{Y_i^2} \leq  \frac{\frac{1}{\sqrt{d_i}} +
\frac{\epsilon}{2} (\sqrt{d_i}+ \delta)(\mu_i + \delta)^2}{(1- \beta^2 -\delta)}.$$
For if this fails at $s_0 \leq s^*,$ then
$$X_i(s_0) \left(\sum_{j=1}^r X_j^2(s_0) -1 \right) <
\frac{\frac{1}{\sqrt{d_i}} +
\frac{\epsilon}{2} (\sqrt{d_i}+ \delta)(\mu_i + \delta)^2}{(1- \beta^2 -\delta)}
Y_i(s_0)^2 (\beta^2 -1 + \delta).$$
Hence by (\ref{eqnX}) $X_i^{\prime}(s_0) < 0.$ Moreover, from above,
$Y_i^{\prime}(s_0)>0$. So in fact our lower bound on $X_i / Y_i^2$ holds
on the whole interval $(-\infty, s_0]$ and
it follows that $X_i$ is monotone decreasing on this interval.
 But this
contradicts the fact that $X_i$ tends to zero as $s$ tends to
$-\infty$ and our claim is established.
\end{proof}

\begin{lemma} \label{XYlimit}
For $i > 1$ we have $\lim_{s \rightarrow -\infty} \frac{X_i}{ Y_i^2}
= \frac{\frac{1}{\sqrt{d_i}} + \frac{\epsilon}{2}\sqrt{d_i} \mu_i^2}{1+\beta^2}.$
For $i=1$, the corresponding limit is $\beta/(1-\beta^2)$.
\end{lemma}

\begin{proof}
  The differential equation (\ref{XY}) is of the form
  (\ref{eqnFHK}), with $h = -(1 + \beta^2)$ and $k =
  \frac{1}{\sqrt{d_i}}+ \frac{\epsilon}{2}\sqrt{d_i} \mu_i^2$.
  The desired result now follows from Lemmas
  \ref{FHKlemma} and \ref{XYbounded}.
\end{proof}

\begin{prop} \label{limit-g1}
As $t$ tends to $0$ we have the following limiting values:
\[
g_1 (0) =0 \;\; : \;\; \dot{g}_{1}(0) = 1 \;\; :  \;\; \dot{g}_{i}(0) =0  \; (i>1).
\]
\end{prop}
\begin{proof}
The first limit follows from (\ref{def-tgi}) and the fact that
$\lim_{s \rightarrow -\infty} \frac{W}{Y_1} = 0$.
Since
\begin{equation} \label{eqn-diffgi}
\dot{g_i} = \sqrt{\lambda_i} (X_i/Y_i),
\end{equation}
the second limit follows from the initial values of $X_1$ and $Y_1$
(recall that $\lambda_1 = d_1 - 1$).
The last limit follows from the above formula, Lemma \ref{XYbounded}
and $\lim_{s \rightarrow -\infty} Y_i = 0$ for $i>1$.
\end{proof}

\begin{prop} \label{gifinite}
For $i > 1$, $\mu_i := \lim_{s \rightarrow -\infty} \frac{W}{Y_i}$ and
$g_i(0)=\sqrt{\lambda_i d_i}\, \mu_i$ are finite and positive.
\end{prop}

\begin{proof}
Finiteness follows from Cor (\ref{WYlim});
it remains to prove positivity. Using (\ref{eqn-diffgi}) and
(\ref{eqnY}) we see that
\[
\frac{g_i^{\prime}}{g_i} =  \frac{X_i}{\sqrt{d_i}}
=\frac{1}{2 \sqrt{d_i}} \frac{X_i}{Y_i^2}
\frac{ (Y_i^2)^{\prime}}{(\sum_{j=1}^{r} X_j^2) - \frac{X_i}{\sqrt{d_i}} -\frac{\epsilon}{2} W^2}.
\]
Integrating this over an interval $(s, s^*)$ using the upper bound for
$X_i/Y_i^2$ in the proof of Lemma \ref{XYbounded} and the lower bound
$\sum_{j=1}^r X_j^2 - \frac{X_i}{\sqrt{d_i}} - \frac{\epsilon}{2} W^2  > \beta^2 - \delta$
we get a bound
\[
\frac{g_i(s^*)}{g_i(s)} \leq \exp (C_i(Y_i^2(s^*)- Y_i^2(s)) )
\]
where $C_i$ is a positive constant. This gives the desired positive
lower bound for $g_i(s)$.
\end{proof}

\begin{lemma} \label{lim-W}
As $s$ tends to $-\infty$, $e^{-\beta^2 s}W $ tends to a finite positive
limit.
\end{lemma}
\begin{proof}
Letting $F = e^{-\beta^2 s} W$,  we see that
\[
F^{\prime} = F \left( \sum_{j=1}^{r} X_j^2  - \beta^2 -\frac{\epsilon}{2} W^2 \right).
\]
Near $s=-\infty$, the  dominant terms in the bracket are $-\frac{\epsilon}{2} W^2$
and
$(X_1+ \beta) (X_1 - \beta) \approx 2\beta (X_1 -\beta),$
which is negative since $\Ha < 1$ (Prop \ref{HQbounds}) implies that
$X_1 < 1/\sqrt{d_1}= \beta$.
So $F$ is monotonic decreasing for large
negative $s$.

The proposition is proved if we can show that $F$ is bounded
from above near $s=-\infty$.

For a sufficiently small $\delta> 0$ ($\delta < \frac{2}{3}\beta^2$
suffices for our purposes), we have $\sum_{j=1}^r X_j^2 -
\frac{\epsilon}{2} W^2 > \beta^2 - \delta$ on some interval $(-\infty,
s^*]$. Eqn.(\ref{eqnW}) then gives the inequality $(W^2)^{\prime} \geq
2(\beta^2 - \delta) W^2$ and we deduce that
$$|W| \leq A_0 \, e^{(\beta^2 - \delta)s}$$
on such an interval. By choosing a smaller $s^*$ if necessary, we can similarly
ensure that
$$ |X_1 - \beta| \leq A_1 \, e^{2(\beta^2 -\delta)s}, \ \ \ \ \ \
|X_1 + \beta| \leq 2\beta + \delta, $$
$$  |X_j(s)| \leq A_j \, e^{2(\beta^2 -\delta)s} \, \, \, j \geq 2, $$
where $A_j, \, 0 \leq j \leq r$ are appropriate positive constants.
If we now integrate the equation for $F$ over $[s, s^*]$ and apply the
above estimates, we obtain an upper bound for $F(s)$.
\end{proof}

\begin{prop} \label{lim-u}
$u$ tends to a finite limit as $t$ tends to $0$.
\end{prop}
\begin{proof} By integrating (\ref{def-u}) we obtain
\begin{equation} \label{u0}
e^{u(0)} = ({\rm positive \; constant })
\prod_{i=2}^{r} \, g_i(0)^{d_i} \left(\lim_{s \rightarrow
-\infty} e^{- \beta^2 s} g_1 \right)^{d_1}.
\end{equation}
The last limit, by (\ref{def-tgi}), is
$\lim_{s \rightarrow -\infty} \left( \sqrt{d_1 \lambda_1}(W/Y_1) e^{-\beta^2 s} \right),$
which is finite and positive by Lemma \ref{lim-W}.
\end{proof}

We shall next record some formulas that will be useful when we
study the limiting values of the second and third derivatives of
$g_i$. They are obtained from (\ref{eqn-diffgi}) by straightforward
computation.

\begin{equation} \label{gidotdot}
\ddot{g_i} = \frac{\lambda_i}{g_i Y_i^2} \left(X_i^2 + Y_i^2
    -\sqrt{d_i}X_i + \frac{\epsilon}{2}\, d_i W^2 \right)
\end{equation}
\begin{equation} \label{dddotgi}
\frac{d^3 g_i}{dt^3} = \frac{\lambda_i}{g_i Y_i^2 W}\left(
    \frac{1}{\sqrt{d_i}}\, X_i (X_i^2 + Y_i^2) -3X_i^2 - Y_i^2 +
    \sqrt{d_i}X_i \left(1+ \sum_{j=1}^r \, X_j^2 \right)
     + \frac{\epsilon}{2}W^2 (2\sqrt{d_i} X_i -d_i) \right)
\end{equation}

We also need to obtain certain limits before
proceeding further with smoothness considerations.

\begin{lemma} \label{rho}
The limit $\rho:= \lim_{s \rightarrow -\infty} \frac{Y_1^2 + \sum_{j=1}^r X_j^2 -1}{W^2}$
exists and is finite.
\end{lemma}

\begin{proof}
By the conservation law (\ref{cons}), we have
\begin{eqnarray*}
\frac{Y_1^2 + \sum_{j=1}^r X_j^2 -1}{W^2} &=&
   \frac{-\sum_{j=2}^r Y_j^2 + \epsilon u W^2 + (C -(n-1)\frac{\epsilon}{2}) W^2}{W^2} \\
   &=& -\sum_{j=2}^r \left(\frac{Y_j}{W} \right)^2 + \epsilon u + C-(n-1)\frac{\epsilon}{2}.
\end{eqnarray*}
The last expression has a finite limit as $s$ approaches $-\infty$ by
Props \ref{gifinite} and \ref{lim-u} and our choice of trajectory.
Note that the limit is negative since for our trajectories ${\mathcal L} <0$ so
$Y_1^2 + \sum_{j=1}^{r} X_j^2 -1 <0$.
\end{proof}

\begin{lemma} \label{XbWfin}
The quantity $\frac{X_1 - \beta}{W^2}$ cannot tend to $-\infty$
as $s$ tends to $-\infty$.
\end{lemma}

\begin{proof}
As we saw in the proof of Lemma \ref{lim-W}, $X_1 < \beta$, so
$\frac{X_1 - \beta}{W^2} < 0$. Let us consider
\begin{equation} \label{Ybhat}
\frac{(Y_1 - \hat{\beta})^{\prime}}{2WW^{\prime}}=
\frac{Y_1}{2 \left(\sum_{j=1}^{r} X_j^2  - \frac{\epsilon}{2}W^2 \right)}
\left( \sum_{j=2}^{r} \frac{X_j^2}{W^2}- \frac{\epsilon}{2} + \frac{X_1 (X_1 - \beta)}{W^2} \right).
\end{equation}
The term outside the bracket on the right-hand side tends to
$\hat{\beta}/2\beta^2$, and the first term in the bracket tends to
zero by Lemma \ref{XYbounded} and Prop \ref{gifinite}. So if
$\frac{X_1 - \beta}{W^2}$ tends to $-\infty$, so does
$\frac{(Y_1 - \hat{\beta})^{\prime}}{(W^2)^{\prime}}$, and hence,
by L'H\^{o}pital's rule, so does $\frac{Y_1 - \hat{\beta}}{W^2}$.
But we also have
\[
\frac{\sum_{j=1}^{r} X_j^2 + Y_1^2 - 1}{W^2} =
\frac{ (X_1+\beta) (X_1 - \beta) + (Y_1 + \hat{\beta})(Y_1 - \hat{\beta}) +
 \sum_{j=2}^{r} X_j^2}{W^2}.
\]
As $s$ tends to $-\infty$, the left-hand side tends to a finite limit by
Lemma \ref{rho} while the right-hand side tends to $-\infty$, a contradiction.
\end{proof}

\begin{lemma} \label{XbWlimit}
\[
\lim_{s \rightarrow -\infty} \left( \frac{X_1 - \beta}{W^2} \right) =
\frac{\beta \rho + \frac{\epsilon}{2} \beta (d_1 -1)}{1 + \beta^2}.
\]
\end{lemma}
\begin{proof}
We observe that $(X_1 - \beta)/W^2$ satisfies the differential
equation
\[
\left(\frac{X_1 - \beta}{W^2} \right)^{\prime} =
-\left(\frac{X_1 - \beta}{W^2}\right) \left(1+ \sum_{j=1}^{r} X_j^2 \right)
+ \frac{\beta}{W^2} \left( \sum_{j=1}^{r} X_j^2 + Y_1^2 -1 \right)
+ \frac{\epsilon}{2}(\sqrt{d_1}-X_1) + \epsilon(X_1 - \beta).
\]
By Lemma \ref{rho}, as $s \rightarrow -\infty$, the term
$\frac{\beta}{W^2}(\sum_{j=1}^{r} X_j^2 + Y_1^2 -1)$
on the right-hand side tends to $\rho \beta$ while the sum of the
last two terms
tend to $ \frac{\epsilon}{2} \beta (d_1 -1).$
Now Lemma \ref{FHKlemma} shows that as $s$ tends to
$-\infty$, $\frac{X_1 - \beta}{W^2}$ either tends to $-\infty$
or to the  claimed value. But it cannot tend to $-\infty$ by Lemma \ref{XbWfin},
so we are done.
\end{proof}

We now return to our analysis of the smoothness conditions.

\begin{prop} \label{giddotfin}
$\ddot{g_i}(0)$ is finite for $i>1$ and is zero if $i=1$.
\end{prop}
\begin{proof}
For $i > 1$ we take the $Y_i^2$ factor in the denominator of (\ref{gidotdot})
and distribute it into the bracket. It is now clear from Lemma \ref{XYlimit}
and Cor \ref{WYlim} that the right-hand side tends to a finite limit as $t$
tends to zero.

If $i=1$, using the conservation law (\ref{cons}) and the fact
$g_1 Y_1 = \sqrt{d_1 \lambda_1} W$, we may rewrite (\ref{gidotdot}) as
\[
\ddot{g_1} = \lambda_1 g_1 \left(\frac{1 - \sqrt{d_1}X_1}
{d_1 \lambda_1 W^2}\right)
+ \frac{g_1}{d_1} \left(\epsilon u + C - (n-1)\frac{\epsilon}{2} \right)
-  \frac{g_1}{d_1} \sum_{j=2}^r \left( \frac{X_j^2}{Y_j^2} + 1\right) \frac{Y_j^2}{W^2}
  + \frac{\epsilon}{2} \sqrt{d_1 \lambda_1} \frac{W}{Y_1}  .
\]
Since $g_1(0)=0$, the first term
on the right-hand side tends to $0$ by Lemma \ref{XbWlimit}. The second
term of the right-hand side tends to zero by Prop \ref{lim-u}. The third
term on the right tends to zero by Lemma \ref{XYbounded} and Prop \ref{gifinite}.
The last term on the right clearly tends to $0$ since $Y_1$ tends to
$\hat{\beta} \neq 0$. Hence $\ddot{g_1}(0) = 0$.
\end{proof}

Similarly we can study the potential $u$. From the relation (\ref{def-u})
we find that

\begin{equation} \label{udot}
\dot{u} = \frac{1}{W} \left(\left(\sum_{j=1}^{r} \sqrt{d_j} X_j\right) -1 \right)
= \sqrt{d_1}\left(\frac{ X_1 - \beta}{W^2}\right)W
+ \sum_{i=2}^{r} \sqrt{d_i}\left(\frac{X_i}{Y_i^2}\right) \left(\frac{Y_i^2}{W^2}\right)W ,
\end{equation}
Hence $\dot{u}(0) = 0$ by Lemmas \ref{XbWlimit}, \ref{XYbounded} and
Prop \ref{gifinite}.

Upon differentiating (\ref{udot}) we obtain
\begin{equation} \label{udotdot}
\ddot{u} = \sum_{i=1}^{r} \, d_i \frac{\ddot{g_i}}{g_i} -\frac{\epsilon}{2}
= \sum_{i=1}^{r} \frac{d_i \lambda_i }{g_i^2 Y_i^2} \left(X_i^2 + Y_i^2
  -\sqrt{d_i}X_i + \frac{\epsilon}{2} d_i W^2 \right) - \frac{\epsilon}{2}.
\end{equation}
In view of Props \ref{gifinite} and \ref{giddotfin}, it remains to
see why $\ddot{g_1}/g_1$ tends to a finite value as $t \rightarrow 0$.
This follows from the argument in the proof of Prop \ref{giddotfin}.

\medskip
So our soliton $(g, u)$ is of class $C^2$. Finally, we will analyse
the limit of third derivatives of $g_i$ as $t$ approaches $ 0$.

If $i=1$, we can rewrite (\ref{dddotgi}) to obtain
\begin{eqnarray*}
\frac{d^3g_1}{dt^3} &=& \sqrt{\frac{\lambda_1}{d_1}} \, Y_1
     \left( \frac{\epsilon}{2}\left(\frac{2\sqrt{d_1}X_1 -d_1}{Y_1^2} \right) +
\sqrt{d_1} \left(\frac{X_1}{Y_1^2}\right) \sum_{j=2}^r X_j \, \frac{X_j}{Y_j^2} \left(\frac{Y_j}{W} \right)^2
\right.\\
   & & \left. + \sqrt{d_1} \left(\frac{X_1}{Y_1^2}\right) \left(\frac{1+ X_1^2}{W^2}\right)
        + \frac{1}{W^2}\left(-\frac{3X_1^2}{Y_1^2} + \beta X_1 +  \frac{\beta X_1^3}{Y_1^2} -1 \right) \right).
\end{eqnarray*}
The first term within the big bracket on the right-hand side certainly
tends to a finite limit and the second term tends to $0$ by Lemma \ref{XYbounded}
and Prop \ref{gifinite}. Now the analogous argument to that in the steady case
(two paragraphs before the statement of Theorem 4.17 in \cite{DW2})
with $\mathcal L$ replaced by $W^2$ shows that the third and fourth term
tend to a finite limit.

If $i > 1$, using (\ref{def-Yi}) in (\ref{dddotgi}) and the initial
conditions $Y_i(0)=0$, we see that in order for the third derivatives
to have $0$ as a limit, it suffices to show that
$$ \frac{1}{W^2}\left( -3 \frac{X_i^2}{Y_i^2} -1  + \frac{X_i}{\sqrt{d_i}}\left(1+\frac{X_i^2}{Y_i^2}\right)
    + \sqrt{d_i}\, \frac{X_i}{Y_i^2} \left(1+ \sum_{j=1}^r \, X_j^2 \right)
     + \frac{\epsilon}{2} \frac{W^2}{Y_i^2} (2\sqrt{d_i} X_i -d_i) \right)
$$
has a finite limit as $s \rightarrow -\infty$. Again by Lemma \ref{XYlimit}
and Prop \ref{gifinite}, it follows that $\frac{X_i}{W^2}$ and hence
$\frac{X_i^2}{Y_i^2 W^2}$ have finite limits, so we are reduced to showing that
\[
\frac{1}{W^2} \left( -\frac{1}{\sqrt{d_i}} + \frac{X_i}{Y_i^2} (1 + X_1^2 )
  -\frac{\epsilon}{2}\sqrt{d_i} \frac{W^2}{Y_i^2} \right)
\]
has a finite limit.

Now  $1 + X_1^2 = 1 + \beta^2$ and $W/Y_i = \mu_i$ modulo terms
which tend to a finite limit when divided by $W^2$ (for the second case
we use L'Hopital's rule).  So we just have to check that
\[
Q_i:=  \frac{1}{W^2} \left(\frac{X_i}{Y_i^2} -
  \frac{\frac{1}{\sqrt{d_i}}+ \frac{\epsilon}{2}\sqrt{d_i} \mu_i^2}{1 + \beta^2}\right)
\]
has a  finite limit.

Let $\tilde{\mu}_i :=\frac{\frac{1}{\sqrt{d_i}}+ \frac{\epsilon}{2}\sqrt{d_i} \mu_i^2}{1 + \beta^2}.$
Proceeding as in the steady case, we find that $Q_i$ satisfies the equation
\begin{eqnarray*}
 Q_i^{\prime} &=& \left(-1 - 3 \sum_{j} X_j^2  + 2\frac{X_i}{\sqrt{d_i}} + \epsilon W^2 \right) Q_i \\
   & & + \left[\frac{\tilde{\mu}_i}{W^2}\left(-1-\sum_{j=1}^r X_j^2 + 2\frac{X_i}{\sqrt{d_i}} \right)
   + \frac{1}{W^2} \left(\frac{1}{\sqrt{d_i}} + \frac{\epsilon}{2} \frac{W^2}{Y_i^2}(\sqrt{d_i} + X_i)
 \right) \right].
\end{eqnarray*}
The coefficient of $Q_i$ on the right-hand side tends to $-1-3\beta^2$ as $s \rightarrow -\infty$.
Using arguments similar to those above (cf. the arguments in the steady case
three paragraphs before the statement of Theorem 4.17 in \cite{DW2}),
 we can also check that the term
in square brackets on the right tends to a finite positive limit.
So the hypotheses of Lemma \ref{FHKlemma} are satisfied, and $Q_i$ either
tends to a finite limit or to $+\infty$ or $-\infty$.  But one can also compute that
\[
\frac{\left(\frac{X_i}{Y_i^2} - \tilde{\mu}_i \right)^{\prime}}{(W^2)^{\prime}}
= \frac{\left(\frac{X_i}{Y_i^2} - \tilde{\mu}_i \right)}{W^2}
   \left( \frac{-(1 + \beta^2)}{2 \left(\sum_{j=1}^{r} X_j^2
- \frac{\epsilon}{2}W^2 \right)} \right)+ R_i
\]
where $R_i$ tends to a finite limit. So
if the limit of $Q_i$ is infinite, L'H\^{o}pital's rule gives a
contradiction, as the term in the final bracket has a negative limit.

\medskip
We have now shown that the metric is $C^3$ and so by the discussion on
regularity near the beginning of \S 3 in \cite{DW2}, the soliton is smooth.
To summarise we may now state the

\begin{thm} \label{mainthm}
Let $M_2, \cdots, M_r$ $($$r \geq 1$$)$ be complete Einstein manifolds with
positive scalar curvature. For $d_1 > 1$ there is an $r$-parameter family
of smooth complete expanding gradient Ricci solitons on the trivial rank
$d_1 + 1$ vector bundle over $M_2 \times \cdots \times M_r$. \qed
\end{thm}

\begin{rmk} \label{Bohm}
Recall that C. B\"ohm proved in \cite{B} that there is also an $r-1$-parameter
family of complete Einstein metrics with negative scalar curvature
on the above manifolds. These metrics correspond to trajectories which lie
in the invariant submanifold $\{\Qa = 0, \Ha =1\}$ and which converge to
the point $E_+$ (cf Remark \ref{einstein}).
When $r=1$, the invariant submanifold is a curve, and this is precisely the
hyperbolic trajectory mentioned in \cite{Cetc}, corresponding to the hyperbolic
metric on $\R^{d_1+1}$. It follows easily from our equations that the mean curvature
of the hypersurfaces tends to the constant $\sqrt{n \epsilon/2}$ as $t \rightarrow
+\infty$ and the metric coefficients $g_i^2$
 grow exponentially fast (compare Remark \ref{asymp}).

Our proof of Theorem \ref{mainthm}, with appropriate minor modifications, then
gives an alternative proof of B\"ohm's result. (In particular, since we are
in the Einstein case, we no longer need to consider the third order derivatives
in the regularity analysis.)
\end{rmk}

Finally, recall that the metrics of our steady solitons in \cite{DW2} have non-negative
Ricci curvature. The analogous fact in the expanding case is

\begin{prop} \label{concave}
The soliton potentials $u$ in the examples in Theorem \ref{mainthm}
are concave, indeed strictly concave off the zero section. Hence
$\ric(g) + \frac{\epsilon}{2}g$ is positive semi-definite. In particular,
$-u$ is subharmonic.
\end{prop}
\begin{proof} It suffices to show that $\frac{\dot{g}_i}{g_i}\dot{u}$ and $\ddot{u}$
are negative for $0< t < +\infty$.  The first fact follows from (\ref{logdiff-gi}),
(\ref{udot}), the positivity of $W, X_i$ and Prop \ref{HQbounds}. To see the second,
note that $\ddot{u} = (\Qa +1 -\Ha)/W^2$ by (\ref{udotdot}) and consider the equation
$$ (\Qa +1 - \Ha)^{\prime} =
(\Qa + 1 - \Ha)(G- 1 - \epsilon W^2) + \Qa G + \frac{\epsilon}{2}(\Ha -1) W^2,$$
which follows from (\ref{eqnH})-(\ref{eqnQ}) and where $G = \sum_i X_i^2$.
By Prop \ref{HQbounds}, on the right-hand side, the second factor of the first
term and the second and third terms  are all negative. So if $\Qa + 1 -\Ha \geq 0$ at some $s_0$,
it has to be strictly decreasing on the left of $s_0$. This contradicts the fact
that $\Qa + 1 - \Ha$ tends to $0$ as $s \rightarrow -\infty$.
\end{proof}

\end{document}